\documentclass[12pt, reqno]{amsart}

\usepackage{latexsym}
\usepackage{amsmath}
\usepackage{amsfonts}
\usepackage{amssymb}
\usepackage{graphicx}
\usepackage{comment}
\usepackage{enumitem}

\DeclareFontFamily{U}{mathx}{}
\DeclareFontShape{U}{mathx}{m}{n}{ <-> mathx10 }{}
\DeclareSymbolFont{mathx}{U}{mathx}{m}{n}
\DeclareFontSubstitution{U}{mathx}{m}{n}

\DeclareMathAccent{\widecheck}{0}{mathx}{"71}

\newtheorem{theorem}{Theorem}[section]
\newtheorem{lemma}[theorem]{Lemma}
\newtheorem*{theorem*}{Theorem}
\newtheorem{corollary}[theorem]{Corollary}
\newtheorem{proposition}[theorem]{Proposition}
\newtheorem* {question*}{Question}

\newtheorem{sublemma}{}[theorem]
\newtheorem{conjecture}[theorem]{Conjecture}

\theoremstyle{definition}

\theoremstyle{remark}

\numberwithin{equation}{section}

\usepackage{mathtools}

\newcommand{\si}{{\rm si}}
\newcommand{\cl}{{\rm cl}}

\newcommand{\exim}{\text{EXIM}}
\newcommand{\hatM}{\widehat{\mathcal{M}}}
\newcommand{\veeM}{\widecheck{\mathcal{M}}}
\newcommand{\tM}{\widetilde{\mathcal{M}}}

\begin{document}

\title[Binary matroids with few excluded induced minors]{Classes of binary matroids with small lists of excluded induced minors}

\author{James Dylan Douthitt}
\address{Mathematics Department\\
Louisiana State University\\
Baton Rouge, Louisiana}
\email{jdouth5@lsu.edu}

\author{James Oxley}
\address{Mathematics Department\\
Louisiana State University\\
Baton Rouge, Louisiana}
\email{oxley@math.lsu.edu}

\subjclass{05B35, 05C75}
\date{\today}

\begin{abstract}
In earlier work, we characterized the class of matroids with no $M(C_4)$ as an induced minor and the class of matroids with no member of $\{M(C_4),M(K_4)\}$ as an induced minor. In this paper,   for every two matroids in $\{M(C_4),M(K_4\backslash e),M(K_4),F_7\}$, we determine the class of matroids that have neither of the chosen pair as an induced minor. Additionally, we prove structural lemmas toward characterizing the class of matroids that do not contain $M(K_4)$ as an induced minor.
\end{abstract}

\maketitle

\section{Introduction}
\label{introduction}
The notation and terminology in this paper will follow 
\cite{ox1}. Unless stated otherwise, all graphs and matroids considered here are simple. Thus, every contraction of a set from a matroid is immediately followed by the simplification of the resulting matroid. An \textit{induced restriction} of a matroid $M$ is a matroid $N$ that can be obtained from $M$ by restricting to a flat. An \textit{induced minor} of $M$ is a matroid $N$ that can be obtained from $M$ by a sequence of restrictions to flats and contractions, where each such contraction is followed by a simplification. Equivalently, N can be obtained from M by at most one restriction to a
flat and at most one contraction followed by a simplification, where
these operations can be performed in either order. As noted by Thomas
Zaslavsky (private communication), this means that $M$ has $N$ as an
induced minor if and only if $N$ is the matroid corresponding to an
interval in the lattice of flats of $M$

Given a set $\mathcal{M}$ of binary matroids, we write $\exim(\mathcal{M})$ for the class of binary matroids with no member of $\mathcal{M}$ as an induced minor. In previous work \cite{douthox1, douthox2}, we characterized $\exim(M(C_4))$ and proved the following characterization of $\exim(M(C_4),M(K_4))$.

\begin{theorem}\label{gf(q)-chordal}
    The class $\exim(M(C_4),M(K_4))$ consists of the binary matroids that can be obtained from projective geometries over $GF(2)$ by repeated generalized parallel connections across projective geometries over $GF(2)$.
\end{theorem}

In Section~\ref{eximSec}, we consider the other five pairs of rank-3 connected binary matroids and describe the binary matroids in which neither member of the pair occurs as an induced minor.  The characterizations of each of these classes include the following.

\begin{theorem}\label{noK4eK4}
    The connected members of $\exim(M(K_4\backslash e),M(K_4))$ consist of all projective geometries, all affine geometries of rank at least three, and all circuits with at least three elements.
\end{theorem}

\begin{theorem}\label{noK4eF7}
    The connected members of $\exim(M(K_4\backslash e),F_7)$ consist of $F_7^*$ and all of the matroids $M(C_n)$ and $M(K_n)$ with $n\geq 3$.
\end{theorem} 

\begin{theorem}\label{noC4K4e}
    The connected members of $\exim(M(C_4),M(K_4\backslash e))$ consist of all projective geometries, all projective geometries with a single point deleted, and all cycle matroids of complete graphs.
\end{theorem}

Let $N$ be a simple $GF(q)$-representable matroid. Oxley and Whittle \cite{oxwhittle} defined the \textit{$q$-coning}, $A(N)$, of $N$ as the $GF(q)$-representable matroid that is obtained by adding a coloop $p$ to $N$ and then adding every point on each line between $p$ and a point of $N$. We call $p$ the \textit{tip} (or \textit{apex}) of the cone. This construction was originally introduced by Whittle \cite{whittle}, who called the operation a \textit{$q$-lift}. Whenever the field is clear, we will write `coning' in place of `$q$-coning'. 
Similarly, a \textit{tipless coning} of $N$ is the matroid $A(N)\backslash p$. We shall sometimes call $A(N)$ a \textit{tipped coning} of $N$.
For a class $\mathcal{M}$ of $GF(q)$-representable matroids, we write $\hatM$ for the class of matroids obtained from $\mathcal{M}$ by repeated $q$-conings. In Section~\ref{conings}, we prove the following result.

\begin{theorem}\label{coneIM}
    If $\mathcal{M}$ is a class of $GF(q)$-representable matroids closed under taking induced minors, then the class $\hatM$ is closed under taking induced minors.
\end{theorem}

In Section~\ref{noK4}, we show that the class of binary matroids that do not contain $M(K_4)$ as an induced minor is closed under taking generalized parallel connections across projective geometries, coning with a tip, and, for a member of the class having no triangles, coning without a tip. Our goal is to eventually prove the following conjecture. A matroid is \textit{triangle-free} if it has no three-element circuits. 

\begin{conjecture}
    The class of $3$-connected binary matroids that do not contain $M(K_4)$ as an induced minor is exactly the class of matroids that can be obtained by starting with binary projective geometries and circuits and applying sequences of the following operations:
    \begin{enumerate}[label=(\roman*)]
        \item generalized parallel connections across projective geometries;
        \item tipped coning; and
        \item tipless coning of triangle-free matroids.
    \end{enumerate}
\end{conjecture}

In Section~\ref{noK4}, we also  prove various structural results for matroids that do not contain $M(K_4)$ as an induced minor. The most important of these, which may be of independent interest, is the following. 

\begin{theorem}\label{trianglethroughe}
    Let $M$ be a $3$-connected matroid having an element $e$ such that neither $M$ nor $M\backslash e$ has $M(K_4)$ as an induced minor. Then every element of $E(M)-\{e\}$ is in a triangle with $e$. 
\end{theorem}

\section{Small sets of excluded induced minors}\label{eximSec}
There are exactly four connected rank-3 binary matroids, namely $M(C_4),M(K_4\backslash e),M(K_4),$ and $F_7$. In this section, we characterize the classes of matroids with pairs of such matroids as excluded induced minors. Unless otherwise stated, the matroids considered in this section will be binary. Recall that, given a set $\mathcal{M}$, of binary matroids, we write $\exim(\mathcal{M})$ for the class of binary matroids with no member of $\mathcal{M}$ as an induced minor. Following Cordovil, Klein, and Forge \cite{cordovil}, a matroid $M$ is \textit{chordal} if, for each circuit $C$ with at least four elements, $\cl_M(C)-C\neq \emptyset $. Similarly, an element $g$ of $E(M)$ is a \textit{chord} of a circuit $C$ if $g\in\cl_M(C)-C$. In \cite{douthox2}, we proved the following characterization of chordal binary matroids.

\begin{lemma}\label{chordal}
    A binary matroid is chordal if and only if it has no $M(C_4)$ as an induced minor.
\end{lemma}

We use this lemma in the proof of the next result.

\begin{lemma}\label{eximC4F7}
    The class $\exim(M(C_4),F_7)$ is exactly the class of chordal regular matroids.
\end{lemma}
\begin{proof}
    By Lemma~\ref{chordal}, $\exim(M(C_4),F_7)$ contains all chordal regular matroids. Let $M$ be in $\exim(M(C_4),F_7)$. Then $M$ has does not have $M(C_4)$ as an induced minor and, by Lemma~\ref{chordal}, $M$ is chordal. Suppose $M$ is not regular. Since $M$ is binary, it does not have $U_{2,4}$ as a minor.  As $M$ does not have $F_7$ as an induced minor, it does not have $F_7$ as a minor. Therefore, since $M$ is not regular, it must have $F_7^*$ as a minor. As $F_7^*$ has $M(C_4)$ as an induced minor, $M$ does not have $F_7^*$ as an induced minor. By the Scum Theorem, $M$ has as a contraction a rank-4 proper extension of $F_7^*$. But every rank-4 proper extension of $F_7^*$ has $F_7$ as an induced minor, a contradiction.
\end{proof}

A graph $H$ is a \textit{parallel extension} of a graph $G$ if $G=H\backslash f$ for an edge $f$ such that $f$ is in a non-trivial parallel class of $H$. Likewise, a graph $H$ is a \textit{series extension} of a graph $G$ if $G=H/f$ for an edge $f$ of $H$ such that $f$ is in a non-trivial series class of $H$. A graph $G$ is a \textit{series-parallel network} if it can be obtained from a loop or $K_2$ by a sequence of operations each of which is a series or parallel extension. A matroid $M$ is called a series-parallel network if $M$ is the cycle matroid of a graph $G$ that is a series-parallel network.
Observe that this means that every series-parallel network is a connected matroid.

\begin{lemma}\label{eximK4F7}
    The connected members of $\exim(M(K_4),F_7)$ consist  precisely of all simple series-parallel networks.
\end{lemma}
\begin{proof}
    Suppose $M$ is a series-parallel network. Then, by a theorem of Dirac \cite{dirac2}, $M$ does not have $M(K_4)$ as a minor and therefore $M$ has neither $M(K_4)$ nor $F_7$ as an induced minor.

    Now suppose $M$ is in $\exim(M(K_4),F_7)$. Since $M$ is binary, $M$ has no $U_{2,4}$-minor. Moreover, $M$ has no $M(K_4)$-minor. Thus, by,   for example,  \cite[Corollary 12.2.14]{ox1},  $M$ is isomorphic to $M(G)$ for some series-parallel network $G$.
\end{proof}

Theorem~\ref{gf(q)-chordal} and Lemmas~\ref{eximC4F7}~and~\ref{eximK4F7} determine $\exim(N_1,N_2)$ for three pairs of matroids from $\{M(C_4),M(K_4),F_7\}$. Each of the remaining pairs $\{N_1,N_2\}$ contains $M(K_4\backslash e)$. The next result will be useful in dealing with these possibilities.

\begin{lemma}\label{always3C}
    Let $M$ be a connected binary matroid that is not $3$-connected and suppose that $|E(M)|\geq 5$. Then 
    \begin{enumerate}[label=(\roman*)]
        \item $M$ is a circuit and has $M(C_4)$ as an induced minor; or
        \item $M$ decomposes as a parallel connection of simple matroids and has $M(K_4\backslash e)$ as an induced minor; or
        \item $M$ decomposes as a $2$-sum of simple matroids and has both $M(C_4)$ and $M(K_e\backslash e)$ as   induced minors.
    \end{enumerate}
\end{lemma}

\begin{proof}
    If $M$ is a circuit, then  $M$ has $M(C_4)$ as an induced minor. As $M$ is connected but not 3-connected, it has a 2-separation $(X,Y)$.
  Suppose that $\cl(X)\cap \cl(Y)=\emptyset$. Then, by \cite[Lemma 2.3]{douthox2}, $M$ has $M(C_4)$ as an induced minor. Moreover, $M$ is the 2-sum of simple matroids $M_1$ and $M_2$ that have ground sets $X\cup \{p\}$ and $Y\cup \{p\}$, respectively.

    Now assume that $M$ is not a circuit. Then, at least one of $M_1$ or $M_2$, say $M_2$, is not a circuit.  Let $C_2$ be a minimum-sized circuit of $M_2$ containing $p$, and let $D_2$ be a circuit of $M_2$ that meets $C_2$ such that $D_2  - C_2$ is minimal and non-empty. 
    The choice of $D_2$ implies that 
    $\{C_2 - D_2, C_2 \cap D_2, D_2 - C_2\}$ is a partition of $C_2 \cup D_2$ such that the union of every two of these sets is a circuit of $M_2$. Moreover, each of the sets in $\{C_2 - D_2, C_2 \cap D_2, D_2 - C_2\}$ is a series class of $M_2|(C_2 \cup D_2)$ and, since $M_2$ is simple, at most one of these sets contains a single element. In $M_2|(C_2 \cup D_2)$, we can contract elements to obtain an $M(K_4 \backslash e)$-minor having $p$ as an element. Thus, $M_2$ has an induced minor $N_2$ using the element $p$ such that $N_2$ is one of $M(K_4 \backslash e)$, $M(K_4)$, or $F_7$. Let $C_1$  be a minimum-sized circuit of $M_1$ containing $p$ and let  $a_1$ and $b_1$ be elements of $C_1 - p$. Then deleting the elements of $M_1$ not in $C_1$ and   contracting  the elements of $C_1 - \{a_1,b_1,p\}$, we get, as an induced minor $N_1$ of $M_1$, a triangle with ground set $\{a_1,b_1,p\}$.
    
 Suppose  $N_2$ is   $M(K_4 \backslash e)$. Then, by taking the 2-sum of $N_1$ and $N_2$, we get an induced minor of $M$ that we can check has $M(K_4 \backslash e)$ as an induced minor. 
 Now suppose $N_2$ is $M(K_4)$ or $F_7$. 
Then we can contract an element of $N_2$ other than $p$ to get a rank-$2$ matroid $N_2'$ having $p$ in a $2$-circuit.   The simplification of the 2-sum of $N_1$ and $N_2'$ with basepoint $p$ is $M(K_4 \backslash e)$, so this matroid is an induced minor of $M$.

We may now assume that $\cl(X)\cap \cl(Y)=\{z\}$ for some element $z$. For each $i$ in $\{1,2\}$, let $C_i$ be a circuit in $M_i$ containing $z$, and let $a_i$ and $b_i$ be distinct elements of $C_i-z$. Then $M/((C_1\cup C_2)-\{a_1,b_1,a_2,b_2,z\})$ has $\{a_1,b_1,z\}$ and $\{a_2,b_2,z\}$ as triangles and has $\{a_1,b_1,a_2,b_2,z\}$ as a flat. Therefore, $M$ has $M(K_4\backslash e)$ as an induced minor. 
\end{proof}

We will often write $P_r$ for the projective geometry $PG(r-1,q)$ when context makes the field clear.

\begin{proof}[Proof of Theorem~\ref{noK4eK4}]
    Let $\mathcal{M}$ be the set of all connected matroids in $\exim(M(K_4\backslash e),M(K_4))$. Clearly, all projective geometries, all affine geometries, and all circuits with at least three elements are in $\mathcal{M}$. 
    
    Let $N$ be in $\mathcal{M}$. If $r(N)=3$, then the $N$ is either $M(C_4)$ or $F_7$ and the result holds. Let $M$ be a smallest-rank member of $\mathcal{M}$ such that $M$ is not a projective geometry, an affine geometry, or a circuit. Let $r(M)=r$. Then $r \ge 4$. By Lemma~\ref{always3C}, we may assume $M$ is 3-connected. By the choice of $M$, the matroid $M/f$ is a projective geometry, an affine geometry, or a circuit for each $f$ in $E(M)$. If $M/f$ is a projective geometry for all $f$ in $E(M)$, then, by \cite[Lemma 3.2]{douthox1}, $M$ is isomorphic to $P_r\backslash P_{i}$ for some $i$ with $0\leq i \leq r-1$. If $1\leq i\leq r-2$, then $M$ has $M(K_4)$ as an induced minor. Therefore, $i=0$ and $M$ is a projective geometry, or $i=r-1$ and $M$ is an affine geometry. In each case, we obtain a contradiction. Therefore, $M$ has an element $f$ such that $M/f$ is $M(C_n)$ for some $n\geq 4$, or $M/f$ is an affine geometry. In each case, there is at most one triangle containing $f$ otherwise $M$ has $M(K_4 \backslash e)$ as an induced minor. Suppose that $M/f\cong M(C_n)$ for some $n\geq 4$. Then $|E(M)|\leq n+2$ and $r(M)=n$. Therefore $r^*(M)\leq 2$, a contradiction since $M$ is 3-connected.

    We now know that $M/f$ is an affine geometry for some $f$ in $E(M)$. In this case, as $f$ is in at most one triangle,  
    \begin{equation}\label{affsize}
        2^{r-2}+1\leq |E(M)|\leq 2^{r-2}+2.
    \end{equation} 
    
    Suppose $M/g$ is a projective geometry for some $g$ in $E(M)$. Then $|E(M/g)|= 2^{r-1}-1$, which contradicts (\ref{affsize}) since $r \ge 4$. Therefore, for all $g$ in $E(M)$, we must have that $M/g$ is an affine geometry. Assume that $M$ is not an affine matroid and let $C$ be a smallest odd circuit in $M$. As $M$ is not a circuit, there is an element $g$ in $E(M) - C$. Since $M/g$ is an affine geometry, $g \in \cl_M(C)$, so $M$ has a circuit $D$ such that $g \in D \subseteq C \cup \{g\}$. Then $D$ or $D \bigtriangleup C$ is an odd circuit that is smaller than $C$, a contradiction. We conclude that $M$ is an affine matroid. Thus, if $g \in E(M)$, then 
$M/g$ has an odd circuit, a contradiction. 
\end{proof}

\begin{lemma}\label{noK4e}
    The $2$-connected simple graphs that do not contain $K_4\backslash e$ as an induced minor are cycles and cliques.
\end{lemma}
\begin{proof}
    Suppose $G$ is a 2-connected graph with no $K_4\backslash e$ as an induced minor.  By Lemma~\ref{always3C}, either $G$ is a cycle, or $G$ is 3-connected.
    Assume that $G$ is not a cycle and not a clique and let $x$ and $y$ be two non-adjacent vertices of $G$ for which the distance $d$ between them is a minimum. Suppose $d\geq 3$. 
 Take a minimum-length $(x,y)$-path    
     and let $x'$ and $y'$ be two adjacent internal vertices of this path where $xx'$ is an edge. Then the distance from $x$ to $y'$ must be 1 by the choice of $x$ and $y$. Hence there is a shorter $(x,y)$-path, a contradiction. Thus $d=2$. Let $w$ be the internal vertex on a length-2 $(x,y)$-path. Since $G$ is 3-connected, $G-w$ is 2-connected. Let $C$ be a shortest cycle of $G$ containing $x$ and $y$ and avoiding $w$. Take $P_1$ and $P_2$ to be the two $(x,y)$-paths such that $E(P_1)\cup E(P_2)=E(C)$ and $|E(P_1)|\leq |E(P_2)|$. If there is an edge from an internal vertex $z$ of $P_1$ or of $P_2$ to $w$, then we may contract $P_1$ or $P_2$ down to two edges, namely $xz$ and $zy$, and obtain $K_4\backslash e$ as an induced minor of $G$, a contradiction. Therefore, there are no edges from $w$ to an internal vertex of $P_1$ or of $P_2$. Then, by contracting $P_1$ to a single edge and $P_2$ to two edges, we obtain $K_4\backslash e$ as an induced minor of $G$, a contradiction. 
\end{proof}

The next two results, which  are due to Hall \cite{hall} (see also \cite[Proposition 12.2.11]{ox1}) and Oxley~\cite[Lemma 3]{ox3}, will be used in the proof of Lemma~\ref{notcograph}.

\begin{theorem}\label{hallres}
    If $G$ is a $3$-connected graph, then $G$ has no $K_{3,3}$-minor if and only if either $G$ is planar or its associated simple graph is $K_5$.
\end{theorem}

\begin{lemma}\label{m*k33}
There is no simple rank-$4$ regular matroid having $M^*(K_{3,3})$ as a proper restriction.
\end{lemma}

\begin{lemma}\label{notcograph}
    Suppose $M$ is a cographic matroid such that $M$ is not graphic. Then $M$ has $M(K_4\backslash e)$ as an induced minor.
\end{lemma}
\begin{proof}
    Suppose $M$ does not have $M(K_4\backslash e)$ as an induced minor. By Lemma~\ref{always3C}, we may assume that $M$ is $3$-connected. Since $M$ is cographic but not graphic, $M$ is the bond matroid of some non-planar graph $G$. Then $G$ has $K_{3,3}$ or $K_5$ as a minor. Suppose $G$ has $K_{3,3}$ as a minor and hence $M$ has $M^*(K_{3,3})$ as a minor. If $M$ does not have $M^*(K_{3,3})$ as an induced minor, then, by Lemma~\ref{m*k33}, $M$ is not regular, a contradiction. Therefore, $M$ has $M^*(K_{3,3})$ as an induced minor,  so $M$ has $M(K_4\backslash e)$ as an induced minor, a contradiction. It now follows, by Theorem \ref{hallres}, that the simple graph associated with $G$ is $K_5$. Since $M^*(G)\cong M$, in forming an induced minor of $M$, any deletion of elements of $G$ is allowed.   Therefore, we may assume $G\cong K_5$. However, by contracting a triangle of $K_5$ and deleting one edge from each of the resulting parallel classes, we obtain the planar dual of $K_4\backslash e$. Hence, by deleting a triad of $M$ and contracting one element from each of the resulting non-trivial series classes, we obtain $M(K_4\backslash e)$ as an induced minor of $M$, a contradiction.
\end{proof}

We will  use  the following result of Seymour \cite{seymour3} (see also \cite[Corollary 12.2.6]{ox1}).

\begin{lemma}\label{seymournoF7}
    Every binary matroid with no $F_7$-minor can be obtained from regular matroids and copies of $F_7^*$ by a sequence of direct sums and $2$-sums.
\end{lemma}
The next result is known as Seymour's Decomposition Theorem \cite{seymour2} (see also \cite[Theorem 13.1.1]{ox1}).
\begin{theorem}\label{seymourdecomp}
    Every regular matroid $M$ can be constructed by using direct sums, $2$-sums, and $3$-sums starting with matroids each of which is either graphic, cographic, or isomorphic to $R_{10}$ and each of which is a minor of $M$.
\end{theorem}

We now prove the second main theorem of this section.

\begin{proof}[Proof of Theorem \ref{noK4eF7}]
    Certainly all circuits, all cycle matroids of complete graphs, and $F_7^*$ are contained in the set of connected members of $\exim(M(K_4\backslash e), F_7)$. 

    Now suppose $M$ is a connected member of $\exim(M(K_4\backslash e),F_7)$. Assume $M$ is not a circuit, the cycle matroid of a complete graph, or $F_7^*$. By Lemma~\ref{always3C}, we may assume $M$ is 3-connected. Since $M$ does not have $F_7$ as an induced minor, $M$ does not have $F_7$ as a minor and therefore, by Lemma \ref{seymournoF7}, $M$ can be built from regular matroids and copies of $F_7^*$ by $2$-sums. Since $M$ is 3-connected and, by assumption,  $M\not\cong F_7^*$, we deduce $M$ is a regular matroid. 
    Then, by Theorem \ref{seymourdecomp}, $M$ can be obtained from graphic matroids, cographic matroids, and copies of $R_{10}$ by 2-sums or 3-sums, and $M$ has each starting matroid as a minor. Since such a minor can be achieved by contracting and simplifying only, $M$ has each starting matroid as an induced minor. If any of the matroids used to obtain $M$ is isomorphic to $R_{10}$, then $M$ has $M^*(K_{3,3})$ as an induced minor 
    since $R_{10}/x \cong M^*(K_{3,3})$ for all elements $x$. Therefore, $M$ has $M(K_4\backslash e)$ as an induced minor, a contradiction. If any of the matroids used to obtain $M$ is cographic but not graphic, then, by Lemma~\ref{notcograph},  $M$ has $M(K_4\backslash e)$ as an induced minor, a contradiction. Thus, all of the matroids used to obtain $M$ must be graphic. 
    Any 2-sum or 3-sum of graphic matroids is a graphic matroid. By Lemma~\ref{noK4e}, $M$ is the cycle matroid of either a cycle or  a clique, a contradiction.
\end{proof}

For a matroid $M$ and a positive integer $k$, denote a vertical $k$-separation   by $(X,G,Y)$ where $(X,G,Y)$ is a partition of $E(M)$ with 
$G=\cl(X)\cap \cl(Y)$ and both $(X\cup G,Y)$ and $(X,Y\cup G)$ are vertical $k$-separations of $M$. Recall that a matroid is \textit{round} if, for all positive integers $k$, it has no vertical 
$k$-separations. 

\begin{lemma}\label{eximC4K4eRound}
Each member of $\exim(M(C_4),M(K_4\backslash e))$ is either disconnected or round.
\end{lemma}
\begin{proof}
    Suppose $M$ is a connected member of $\exim(M(C_4),M(K_4\backslash e))$, and that $M$ has a vertical $k$-separation $(X,G,Y)$ for some $k\geq 2$. First suppose $G=\emptyset$. Then, by \cite[Lemma 2.3]{douthox2}, $M$ has $M(C_4)$ as an induced minor,
    a contradiction.

    If $r(G)<k-1$, then $M/G$ has a vertical $k'$-separation $(X',G',Y')$ with $G'=\emptyset$ and, by \cite[Lemma 2.3]{douthox2}, $M/G$ has $M(C_4)$ as an induced minor, a contradiction. We deduce that $r(G)=k-1$. Let $g$ be an element of $G$. Then $M'=M/(G-\{g\})$ is a parallel connection of $(M|\cl(X))/(G-g)$, and $(M|\cl(Y))/(G-g)$. By Lemma~\ref{always3C}, $M$ has $M(K_4\backslash e)$ as an induced minor, a contradiction.
\end{proof}
A flat $F$ in a matroid $M$ is a \textit{connected flat} if $M|F$ is connected.

\begin{corollary}\label{corround}
    If $F$ is a connected flat in a matroid $M$ in the class $\exim(M(C_4),M(K_4\backslash e))$, then $M|F$ is round.
\end{corollary}

McNulty and Wu \cite{mcnulty} proved the following result.

\begin{lemma}\label{mcnultywu}
    Let $M$ be a $3$-connected binary matroid with at least four elements. Then, whenever $f$ and $g$ are distinct elements of $M$, there is a connected hyperplane of $M$ containing $f$ and avoiding $g$.
\end{lemma}

We now prove the third main result of this section. 
\begin{proof}[Proof of Theorem~\ref{noC4K4e}]
    Clearly each of the matroids listed is a connected member of $\exim(M(C_4),M(K_4\backslash e))$. We now show that these are the only connected matroids in $\exim(M(C_4),M(K_4\backslash e))$. 
Let $M$ be a minimum-rank connected member of $\exim(M(C_4),M(K_4\backslash e))$ that is not one of the listed matroids. Then $r(M) \ge 4$.  Assume $r(M)=4$. By Lemma~\ref{always3C}, we may assume that $M$ is $3$-connected. If $M$ is graphic, then $M \cong M({\mathcal W}_4)$ or $M \cong M(K_5 \backslash e)$, so $M$ has $M(K_4 \backslash e)$ as an induced minor, a contradiction. Thus, $M$ is not graphic, so, by a theorem of Tutte~\cite{tutte}, either $M$ is not regular, or $M$ is regular and has $M^*(K_{3,3})$ as a minor. In the second case, by Lemma~\ref{m*k33},  $M \cong M^*(K_{3,3})$, so $M$ has $M(K_4 \backslash e)$ as an induced minor, a contradiction. We deduce that $M$ is not regular.  Then, by another theorem of Tutte~\cite{tutte0}, $M$
has $F_7$ or $F^*_7$ as a minor. By a result of Seymour~\cite{seymour} 
(see also \cite[Lemma 12.2.4]{ox1}), $M$ must have $F_7^*$ as a spanning restriction. As $F_7^*$ has $M(C_4)$ as an induced restriction, $|E(M)| \ge 8$. Thus, by \cite[Lemma 12.2.4]{ox1} again, $M$ has  $AG(3,2)$ or $S_8$ as a restriction. Since the last two matroids have $M(C_4)$ and $M(K_4\backslash e)$ as induced restrictions, we deduce that $|E(M)| \ge 9$. 
Thus, the complement $M^c$ of $M$ in $P_4$ is a proper restriction of $F_7$ or of $M(K_4) \oplus U_{1,1}$. Since $M$ is not $P_4$ or a single-element deletion of $P_4$, we see that $|E(M)| \le 13$. Now $M^c$ has $U_{2,2}$ or $U_{2,3}$ as a rank-2 flat $F$. Each of the three projective planes that contain $F$ must contain a point of $M^c$ that is not in $F$, otherwise $M$ has $M(K_4\backslash e)$ or $M(C_4)$ as an induced restriction. 
But each  of the possibilities for $M^c$ has rank 4 so none is a restriction of $F_7$. Thus, $M^c$ is a  restriction of  $M(K_4) \oplus U_{1,1}$ having at least five points including the point corresponding to $U_{1,1}$. Hence, $M^c$ is $M(K_4\backslash e) \oplus U_{1,1}$, $U_{2,3} \oplus U_{2,2}$, or $M(C_4) \oplus U_{1,1}$. In each case, $M$ has $M(C_4)$ or $M(K_4\backslash e)$ as an induced restriction. We conclude that $r(M) \neq 4$.

We may now assume that $r(M)\geq 5$. By Lemma~\ref{mcnultywu}, $M$ has a connected hyperplane, $H$. By the choice of $M$, it follows that $M|H$ is a projective geometry, 
a projective geometry with a point deleted, or the cycle matroid of a complete graph. 
Suppose first that $M|H$ is a projective geometry with one point deleted. Let $p$ be the projective point missing from $H$ and let $x$ be in $E(M)-H$. Suppose the point $y$ on the line from $x$ to $p$ is in $E(M)$. Then, for each element $h$ in $H$, by the choice of $M$, the plane spanned by $\{x,h,y\}$ is isomorphic to $M(K_4)$. In particular, for each $h$ in $H$, the third point on the line spanned by $\{x,h\}$ is in $E(M)$. Hence, $M|\cl_M(H\cup\{x\})$, which equals $M$,  is a projective geometry with one point deleted, a contradiction. Therefore, $y\not \in E(M)$. Let $z$ be in $E(M)-(H\cup \{x\})$. Then there is a point $g$ of $H$ such that $\{x,z,g\}$ is a triangle. However, $g$ is in a rank-4 flat $F$ of $H$ that spans $p$ and $M|F$ is isomorphic to $P_4-\{p\}$. This implies that $\cl(F\cup x)$ is a connected flat of $M$ and, by the choice of $M$, the matroid $M|\cl(F\cup x)$ is  a projective geometry,   a projective geometry with one point deleted, or the cycle matroid of a complete graph.  Since $M|\cl(F)$ has $F_7$ as a minor, $M|\cl(F\cup x)$ cannot be a the cycle matroid of a complete graph. Moreover, since $p\not\in E(M)$, the matroid $M|\cl(F\cup x)$ is a projective geometry with exactly one point deleted, a contradiction as the point on the line between $x$ and $p$ was assumed to  not be in $E(M)$. We deduce that no connected hyperplane of $M$ is a projective geometry with a point deleted.

Next suppose that $M|H$ is a projective geometry. Let $x$ and $y$ be distinct elements of $E(M)-H$. Then there is an element $h$ in $H$ such that $\{x,y,h\}$ is a triangle of $M$. Since, by Lemma~\ref{eximC4K4eRound}, $M$ is round, there must be an element $z$ in $E(M)-(H\cup \{x,y\})$. This implies there are elements $f_z$ and $g_z$ in $H$ such that $\{x,z,f_z\}$ and $\{y,z,g_z\}$ are triangles of $M$. Then $M|\cl(\{x,y,z,f,g_z,h_z\})$ is either $M(K_4)$ or $F_7$. Let $X$ be a largest subset of $E(M)-H$ containing $\{x,y,z\}$ such that $M|\cl(X)$ is a projective geometry with at most one point deleted. If $X=E(M)-H$, then, since $M$ is round, the cocircuit $X$ is spanning. Thus, $M$ is a projective geometry with at most one point deleted, a contradiction. Let $w$ be in $E(M)-(H\cup \cl(X))$. Consider $M|\cl(X \cup w)$.
For any point $t$ of $X$, there is a point $h_t$ in $H-\cl(X)$ such that $\{t,w,h_t\}$ is a triangle of $M$. Thus, $\cl(X \cup w)$ is a connected flat of $M$. As $r(\cl(X \cup w) \cap H) \ge 3$, we see that $\cl(X \cup w)$ has $F_7$ as a restriction. Hence either $\cl(X \cup w) = E(M)$, or $\cl(X \cup w)$ is a projective geometry or a projective geometry with a point deleted. The second possibility contradicts the choice of $X$. Since $\cl(X \cup w) = E(M)$, we see that $\cl(X)$ is a connected hyperplane of $M$. As $r(\cl(X) \cap H) \ge 3$, we deduce that $\cl(X) \cap H$ has $F_7$ as a restriction, so $\cl(X)$ is a projective geometry of rank $r(M) - 1$.

We show next that $M$ is a projective geometry by showing that each line through $w$ and a point $h$ of $H$ contains three points of $M$. This is certainly true if $h \in H - \cl(X)$ because $\cl(X)$ is a projective hyperplane. Now take $h$ in $\cl(X) \cap H$. Let $x_1$ be a point of $X - H$. Then the third point $x_2$ on the projective line spanned by $\{x_1,h\}$ is in  $E(M)$. Extend $\{x_1,x_2\}$ to a basis $B_X$ of $X$. Take $b$ in $B_X - \{x_1,x_2\}$. Then $\cl((B_X - b) \cup w)$ is a connected hyperplane of $M$ containing $h$ and $w$. Since this hyperplane has an $F_7$ restriction, it must be a projective geometry. Thus, the third point on the projective line spanned by $\{w,h\}$ is in $E(M)$. We conclude that $M$ is indeed a projective geometry, a contradiction. 
 We deduce that no connected hyperplane of $M$ is a projective geometry.

    We now know that, for every connected hyperplane $H$ of $M$, the matroid $M|H$ must be isomorphic to the cycle matroid of a complete graph. Hence $|H|={r(M) \choose 2}$. 
    \begin{sublemma}\label{completegraphsub}
        For all $x$ in $E(M)$, the matroid $M/x$ is the cycle matroid of a complete graph.
    \end{sublemma}
    Suppose, for some $x$ in $E(M)-H$, that $M/x$ is a projective geometry or a projective geometry with exactly one point deleted. As $M|H\cong M(K_n)$ for some $n\geq 5$, we see that $M|H$ has a flat $F$ such that $M|F\cong M(K_5)$. Viewing $M$ as a restriction of $P_{r(M)}$, there are elements $a$ and $b$ of $\cl_{P_{r(M)}}(F)-F$  such that $\{x,a,a_x\}$ and $\{x,b,b_x\}$ are triangles of $P_{r(M)}$ for some elements $a_x$ and $b_x$ in $E(M)-(H\cup \{x\})$ because $M/x$ is a projective geometry possibly with a point deleted.
    
    Since $M|F\cong M(K_5)$, the set $\cl_{P_{r(M)}}(F)-F$ is a $5$-circuit so it does not contain any triangles. Thus, there is a point $y$ in $F$ such that $\{a,b,y\}$ is a triangle of $P_{r(M)}$. 
If the third point on the line $\cl_{P_{r(M)}}(\{x,y\})$ is in $E(M)$, then $\cl_M(\{x,a_x,b_x\})$ is a flat isomorphic to $M(K_4\backslash e)$, a contradiction. Thus, the third point on $\cl_{P_{r(M)}}(\{x,y\})$ is not in $E(M)$. Let $d$ be a point of $\cl_{P_{r(M)}}(F)-(F\cup\{a,b\})$ such that $\{x,d,d_x\}$ is a triangle in $P_{r(M)}$ for some $d_x$ in $E(M)-H$. Note that such an element must exist since $M/x$ is a projective geometry with at most one point deleted and $|\cl_{P_{r(M)}}(F)-(F\cup\{a,b\})|=3$.
    
   As the complement of $M(K_5)$ in $P_4$ is a 5-circuit, the third point $t$ on the line $\cl_{P_{r(M)}}(\{d,y\})$ must be in $F$, otherwise $\{a,b,d,t\}$ forms a 4-circuit in $\cl_{P_{r(M)}}(F)-F$. Thus, $\cl_{M}(\{x,d_x,y\})$ is a connected rank-3 flat containing four or five elements, a contradiction. 
       We conclude that $M/x$ is the cycle matroid of a complete graph for all $x$ in $E(M)-H$. 
    
    Fix $f$ in $E(M)-H$ and suppose, for some $g$ in $E(M)-f$, that $M/g$ is a projective geometry or a projective geometry with one point deleted. In $M/g$, let $f$ label the point corresponding to the parallel class that contained $f$. Then $M/g/f$ is a projective geometry. Since $r(M)\ge5$, it follows that $M/g/f$ has $F_7$ as a minor, a contradiction since $M/f\cong M(K_{r(M)})$. Therefore, \ref{completegraphsub} holds.

By  \ref{completegraphsub}, every single-element contraction of $M$ is regular. Since $M$ is binary of rank at least five, the Scum Theorem implies that $M$ has neither $F_7$ nor $F^*_7$  as a minor. Thus, by a theorem of Tutte~\cite{tutte0}, $M$ is regular. Let $C^*=E(M)-H$ and suppose $C^*$ is dependent. Since $M$ is binary, every circuit contained in $C^*$ has even cardinality. 
Let $\{x_1,x_2,\dots,x_{s}\}$  be a circuit $C$ contained in $C^*$. Because each of $M|H$ and $M/x_{s}$ is the cycle matroid of a complete graph, the third point $y_i$ on the projective line spanned by $\{x_i,x_{s}\}$ is in $M$. Then $M|(C \cup \{y_1,y_2,\dots,y_{s-1}\})$ is isomorphic to a binary spike of rank $s-1$ with a tip. As the last matroid has $F_7$ as a minor,   we have  a contradiction. We conclude that $C^*$ is independent, and, by Lemma~\ref{eximC4K4eRound}, $M$ is round, so $C^*$ is spanning. Therefore, $|C^*|=r(M)$. Hence, 
$|E(M)| = r(M)+{r(M) \choose 2}={r(M)+1 \choose 2}$. By \cite{murty} (see also \cite[Proposition 14.10.3]{ox1}), as $M$ is a regular matroid with ${r(M)+1 \choose 2}$ elements, $M\cong M(K_{r(M)+1})$, a contradiction. 
\end{proof}

\section{Conings}\label{conings}
In this section, we show that, for an induced-minor-closed class $\mathcal{M}$ of $GF(q)$-representable matroids, the class $\hatM$  of matroids that contains $\mathcal{M}$ together with all matroids that can be built from members of   $\mathcal{M}$ by repeatedly  taking $q$-conings  is also an induced-minor-closed class.  Unless specified otherwise, all matroids considered in this section are $GF(q)$-representable. 
\begin{lemma}\label{coneCont}
    If $\mathcal{M}$ is a class of $GF(q)$-representable matroids closed under taking induced minors, then $\hatM$ is closed under taking contractions. 
\end{lemma}
\begin{proof}
Let $M$ be a smallest-rank member of $\hatM$ for which there is an element $e$ such that $M/e \not\in  \hatM$. As $\mathcal{M}$ is closed under taking induced minors, $M \not\in \mathcal{M}$. 
 Thus,  $M=A(N)$ for some $N$ in $\hatM$.  If $e=p$, the tip  of the coning, then $M/e\cong N$, so $M/e  \in \hatM$, a contradiction.  Hence  $e \in E(M)-p$. By \cite[Lemma 2.3(i)]{whittle}, there is a hyperplane $H$ of $M$ that contains $e$ such that $M|H\cong N$. Let $N'=M|H$. Then $N'$ is in $\hatM$ and, therefore, $N'/e$ is in $\hatM$. Thus, by \cite[Lemma 2.3(iii)]{whittle}, we have $M/e\cong A(N'/e)$, 
so $M/e \in \hatM$, a contradiction.
\end{proof}

\begin{lemma}\label{coneIR}
    If $\mathcal{M}$ is a class of $GF(q)$-representable matroids closed under taking induced restrictions, then $\hatM$ is closed under taking induced restrictions.
\end{lemma}

\begin{proof}
Let $M$ be a smallest-rank member of $\hatM$ having a flat $F$ such that $M|F \not\in \hatM$.
 Then $M \not\in \mathcal{M}$, so 
 $M=A(N)$ for some $N$ in $\mathcal{M}$. Suppose $F$ contains the tip $p$. Then $F\cap E(N)$ is a flat of $N$ and, as $N \in 
 \mathcal{M}$ and $r(N) < r(M)$,   it follows that $N|(F\cap E(N)) \in \mathcal{M}$. Therefore, as $M|F=A(N|(F\cap E(N)))$, we see that $M|F$ is in $\hatM$, a contradiction. Thus, $p \not\in F$. Hence, by \cite[Lemma 2.3 (ii)]{whittle}, $M$ has a hyperplane $H$ avoiding $p$ such that $M|H\cong N$ and $F\subseteq H$. Let $N'=M|H$. Then $N' \in \hatM$. As $M|F=N'|F$, we see,  since $r(N') < r(M)$, that $M|F \in \hatM$, a contradiction.
\end{proof}

\begin{proof}[Proof of Theorem~\ref{coneIM}]
    This is an immediate consequence of combining Lemmas~\ref{coneCont}~and~\ref{coneIR}.
\end{proof}

Given a class  $\mathcal{M}$ of $GF(q)$-representable matroids, let $\veeM$ be the class of matroids obtained by starting with matroids in $\mathcal{M}$ and applying a sequence of tipless $q$-conings; let $\tM$ be the class of matroids obtained by starting with matroids in $\mathcal{M}$ and applying a sequence of operations each of which is a tipped $q$-coning   or a tipless $q$-coning.   As an example, 
we have the following result whose straightforward proof is omitted.

\begin{proposition}
\label{tiplesscity}
The  class of $GF(q)$-representable matroids that can be obtained by starting with the empty matroid and applying a sequence of tipless $q$-conings consists of  the class of affine geometries over $GF(q)$. 
\end{proposition}

If $N$ is an affine geometry over $GF(q)$ with $r(N) \ge 3$ and $e \in E(N)$, then $N/e$ is a projective geometry over $GF(q)$ of rank at least two.  
Hence, for a class $\mathcal{M}$ of $GF(q)$-representable matroids, the class $\veeM$ need not be closed under taking contractions. 

The following results are essentially consequences of \cite[Lemma 2.3]{whittle}.  We omit their proofs because they are so similar to the proofs of Lemmas~\ref{coneIR}~and~\ref{coneCont}.

\begin{lemma}\label{tiportiplessIR}
    If $\mathcal{M}$ is a class of $GF(q)$-representable matroids closed under taking induced restrictions, then $\veeM$ is closed under taking induced restrictions.
\end{lemma}

\begin{theorem}\label{tiportiplessIM}
    If $\mathcal{M}$ is a class of $GF(q)$-representable matroids closed under taking induced minors, then $\tM$ is closed under taking induced minors.
\end{theorem}

Recall that we are writing $P_r$ as an abbreviation of $PG(r-1,q)$. For a subset $G$ of $E(P_r)$, let $R= E(P_r) - G$. We view the elements of $G$ and $R$ as being colored green and red, respectively. We call $P_r|G$ a 
\textit{projective target} if there is a sequence $(F_0, F_1, \dots , F_r)$ of projective flats, that is, flats of $P_r$,  with $F_0\subseteq F_1 \subseteq \dots \subseteq F_{r-1}\subseteq F_r$ and $r(F_j) = j$ for all $j$ such that, for all $i$ in $[r]$, the set $F_i - F_{i-1}$ is contained in either $G$ or $R$. 
Such matroids were studied by Nelson and Nomoto \cite{nelson} in the binary case and by Mizell and Oxley \cite{mizell} for $GF(q)$-representable matroids when $q\ge2$. 
Although $GF(q)$-representable matroids need not be uniquely $GF(q)$-representable, it was shown by Mizell and Oxley~\cite[Proposition 6]{mizell} that if one $GF(q)$-representation of a simple $GF(q)$-representable matroid is a target, then all of the $GF(q)$-representations of $M$ are targets. 
Let $M$ be the projective target associated with the sequence $(F_0, F_1, \dots, F_r)$ of projective flats. Then, for any basis $\{x_1,x_2,\dots,x_r\}$ for $P_r$ such that $\{x_1,x_2,\dots,x_i\}$ spans $F_i$ for each $i$, all the elements of
$F_{i+1}-F_i$ have the same color as $x_{i+1}$.

Given a simple $GF(q)$-representable matroid $M$ of rank at most $r$, it is convenient to view $M$ as a restriction of $P_r$. For this, we will take $(G,R)$ to be a partition of $P_r$  such that $P_r|G\cong M$. We can obtain a tipped or tipless coning of $M$ from this 2-coloring $(G,R)$ of $P_r$ by first viewing this $P_r$ as a hyperplane $H$ of $P_{r+1}$ and then taking a point $p$ of $E(P_{r+1}) - H$. Next, for each point $z$ of $H$, we color each point $z'$ of $E(P_{r+1}) - (H \cup p)$ that is on the line between $p$ and $z$ so that the colors of $z$ and $z'$ agree. Finally.
 for a tipped coning, we color $p$ green, and, for a tipless coning, we color $p$ red. It was noted in \cite{oxwhittle} that, in general, two $q$-conings of a matroid need not be isomorphic. Nevertheless, we have the following result.

\begin{theorem}\label{targetsthm}
    The class of $GF(q)$-projective targets is exactly the class of matroids that can be obtained from the empty matroid by a sequence of $q$-conings and tipless $q$-conings.
\end{theorem}

\begin{proof}
    Suppose $M$ is a  projective target of rank at most $r$. As projective targets are uniquely $GF(q)$-representable, we lose no generality in considering a $2$-coloring $(G,R)$ of $E(P_r)$ such that 
 $P_r|G\cong E(M)$. Let $\{x_1,x_2,\dots,x_r\}$ be a basis of $P_r$ such that  $\{x_1,x_2,\dots,x_i\}$ spans a projective flat $X_i$, and the color of $x_{i+1}$ coincides with the color of each of the points in $X_{i+1}-X_{i}$ for each $i$ in  $[r-1]$. Every element $z$ of $M$ can be uniquely written as a linear combination of $x_1,x_2,\dots,x_r$, and the  color  of $z$ matches that of $x_j$ where $j$ is the  highest index  of a basis vector used in this linear combination. For each $i$ in $[r]$, let  $y_i=x_{r-i+1}$. If $y_1$ is green, then $P_r|\{y_1\}$ corresponds to a tipped coning of the empty matroid. If  $y_1$ is red, then $P_r|\{y_1\}$ corresponds to a tipless coning of the empty matroid.  
For each $i \ge 2$, the color of each point $y$ of the flat  $Y_i$ that is spanned by $\{y_1,y_2,\dots,y_i\}$ coincides with the color of $y_j$ where $j$ is the lowest index of a member of $\{y_1,y_2,\dots,y_i\}$ that is used in the  linear combination yielding $y$. Thus, inductively, we see that each flat $Y_i$ may be built by starting with a 2-coloring of the projective flat $Y_{i-1}$ and adding $y_i$ as a coloop and then adding all of the projective points between $y_i$ and the points of $Y_{i-1}$ such that each added point has the same color as its corresponding point in $Y_{i-1}$. We conclude that each projective target of rank at most $r$ can be obtained from the empty matroid by a sequence or $r$ operations each of which is a tipped  or a tipless coning. 

On the other hand, suppose we have a matroid that can be obtained from the empty matroid by a sequence of $r$ operations, each a tipped or tipless coning. 
Let the tips, in order, be $y_1,y_2,\dots,y_r$, where these are colored  green or red if the coning is tipped or tipless, respectively. For each $i$ in $[r]$, let  $x_i=y_{r-i+1}$. Now let $X_i$ be the projective flat that is spanned by $\{x_1,x_2,\dots,x_i\}$. Then $X_1$, which equals $\{y_r\}$, is a tipped or tipless coning of the empty matroid depending on whether $y_r$ is green or red. Moreover, it is not difficult to check that, for all $i$ in $[r]$,  the color of each element of $X_i - X_{i-1}$ coincides with the color of $x_i$. We conclude that  $P_r|G$ is a 
$GF(q)$ projective target. 
\end{proof}

The empty projective target can be associated with the 0-1 string whose sole entry is $0$.

\begin{corollary}\label{01stringtarget}
A non-empty projective target  can be uniquely represented, up to isomorphism, by a $0$-$1$ string whose leftmost entry is a $1$.
\end{corollary}

\begin{proof}
By Theorem~\ref{targetsthm}, non-empty projective targets are exactly the matroids that can be obtained by starting with the empty matroid and repeatedly coning with or without a tip so that at least one coning has a tip. To construct the matroid corresponding to a particular 0-1 string that begins with a 1, we read the string from left to right
interpreting each 1 as a tipped coning and each 0 as a tipless coning. 
\end{proof}

The next result seems unlikely to be new, but we include it for completeness.

\begin{corollary}
    For each non-negative integer $n$, there is,   up to isomorphism,  a unique binary projective target on $n$ elements.
\end{corollary}

\begin{proof}
 As noted above, when $n=0$, we associate the string $0$ with the empty matroid. Now suppose $n\geq 1$. Then $n$ has a binary expansion as a 0-1 string whose leftmost entry is a $1$. By Corollary~\ref{01stringtarget}, there is a unique non-empty binary projective target with this 0-1 string. If $n=1$, then the 0-1 string is $1$, which corresponds to a tipped coning of the empty matroid, so the resulting matroid has exactly one element. Suppose the result holds for 0-1 strings of length less than $k$ that have 1 as their leftmost entry. Take a 0-1 string $S$ of length $k$ having 1 as its leftmost entry. Let $N$ be the unique binary projective target corresponding to the 0-1 string $S'$ obtained by deleting the rightmost entry of $S$. 
Then, by coning $N$, we get a matroid having $2|E(N)|$ elements if the coning is tipless, or having $2|E(N)|+1$ elements if the coning is tipped. These two possibilities correspond to the two choices for $S$, which are obtained from $S'$  by adjoining a 0 or a 1, respectively, as the rightmost entry. 
\end{proof}

\section{Avoiding $M(K_4)$ as an induced minor}\label{noK4}
In this section, we show that the class $\exim(M(K_4))$ is closed under taking generalized parallel connections across projective geometries, coning with a tip, and, for triangle-free matroids, coning without a tip. Unless otherwise stated, all matroids in this section are assumed to be binary. Let $\mathcal{N}$ be the class of binary matroids with that do not contain $M(K_4)$ as an induced minor.

\begin{lemma}\label{K4GPC}
    The class $\mathcal{N}$ is closed under taking generalized parallel connections of members of $\mathcal{N}$ across projective geometries.
\end{lemma}
\begin{proof}  
    Let $M$ be a smallest-rank non-member of $\mathcal{N}$ such that there are matroids $M_1$ and $M_2$ in $\mathcal{N}$ and a projective geometry $N$ with $M=P_N(M_1,M_2)$. Suppose, for some element $e$ of $E(M)$, the matroid $M/e$ has $M(K_4)$ as an induced minor. Then, by symmetry, we may assume that $e \in E(M_1)-E(N)$ or $e \in E(N)$. In the first case, $M/e=P_N((M_1/e),M_2)$, so $M/e$ is a generalized parallel connection of two matroids in $\mathcal{N}$ across a projective geometry. Because $M/e$ has smaller rank than $M$, we conclude that $M/e$ does not have $M(K_4)$ as an induced minor, a contradiction. Now suppose that $e \in E(N)$. Then $M/e=P_{N/e}((M_1/e),(M_2/e))$. Since $N/e$ is a projective geometry and both $M_1/e$ and $M_2/e$ are in $\mathcal{N}$, we deduce that $M/e \in \mathcal{N}$, a contradiction. Hence $M/e \in \mathcal{N}$ for all $e$ in $E(M)$. As $M \not\in \mathcal{N}$, we see that 
 $M$ must have a proper flat $F$ such that $M|F$ is isomorphic to $M(K_4)$. Now $F\cap E(M_1)$ is a flat of $M_1$ and $F\cap E(M_2)$ is a flat of $M_2$. Without loss of generality, we may assume that $E(M_1)$ contains a basis of $M|F$.  Hence, $F$ is contained in $E(M_1)$. Therefore, $M_1$ has $M(K_4)$ as an induced minor, a contradiction as $M_1\in \mathcal{N}$.
\end{proof}

The next two results identify some structure of matroids that have a coning point, that is, a point $f$ such that $f$ is in a triangle with every other element of the matroid.
\begin{lemma}\label{P4-2}
If $N$ has an element $x$ in a triangle with every element and $N$ has a flat $F$ such that $N|F$ is isomorphic to $M(K_4)$, then $N|\cl_N(F\cup \{x\})$ is isomorphic to a rank-$4$ projective geometry with two points deleted. 
\end{lemma}
\begin{proof}
    The element $x$ cannot be in $F$ since every element of $F$ is contained in a rank-2 flat of size 2 but every rank-2 flat of $N$ that contains $x$ is a triangle. Thus, $r(\cl_N(F \cup \{x\})) = 4$ and the result follows since 
every element of $F$ is in a triangle with $x$. 
\end{proof}

\begin{lemma}\label{ExtX}
    Let $M$ be a rank-$r$ matroid in $\mathcal{N}$ and suppose that, for some $x$ in $E(M)$, the matroid $N$ is obtained by adding every element on the line $\cl_{P_r}(\{x,y\})$ for each $y$ in $E(M)-x$. Then $N$ is in $\mathcal{N}$.
\end{lemma}
\begin{proof}
    Suppose that, for some subset $X$ of $E(N)$, the matroid $N/X$ has a flat $F$ such that $(N/X)|F \cong M(K_4)$. Since  $M/x$ does not have $M(K_4)$ as an induced minor, and 
    $N/x \cong M/x$, we may assume $x \in E(N/X)$. Then, by Lemma~\ref{P4-2}, $N/X$ has a flat $F'$ containing $x$ such that $(N/X)|F'$ is isomorphic to a rank-4 projective geometry with two points deleted.  Then, by contracting $x$, we obtain $M(K_4)$ as an induced minor of $M$, a contradiction. 
\end{proof}

Recall that, when $N$ is a simple binary matroid, $A(N)$ is the coning of $N$ with tip $p$, and $A(N)\backslash p$ is the tipless coning of $N$.  

\begin{lemma}\label{tipInClass}
    If $N$ is in $\mathcal{N}$, then $A(N)$ is in $\mathcal{N}$.
\end{lemma}
\begin{proof}
    
    Let $N$ be a smallest-rank member of $\mathcal{N}$ such that $A(N)$ is not in $\mathcal{N}$. Suppose, for some $e$ in $E(A(N))$, the matroid $A(N)/e$ has $M(K_4)$ as an induced minor. 
As $A(N)/p \cong N$, we see that $e \neq p$.    
 Then, by \cite[Lemma 2.3 (i)]{whittle}, $e$ is   in a hyperplane $H$ of $M$ such that $M|H\cong N$. Hence we may assume that $e\in E(N)$. Then $A(N)/e \cong A(N/e)$. By the choice of $N$, the latter  is in $\mathcal{N}$. Hence so is the former, a contradiction.

    We conclude that $A(N)$ contains a flat $F$ such that $A(N)|F \cong M(K_4)$. As $A(N)$ has an element, namely $p$, that is in a triangle with every other element of $E(A(N))$, it follows, by Lemma~\ref{P4-2}, that $A(N)$ has a rank-4 flat $F'$ containing $p$ that is isomorphic to a rank-4 projective geometry with two points deleted. Hence $(A(N)|F')/p$ is isomorphic to $M(K_4)$. Since $p$ was contracted to produce this induced minor, we conclude that $(A(N)|F')/p$ is isomorphic to an induced minor of $N$, a contradiction.
\end{proof}

\begin{lemma}\label{tiplessInClass}
    If $N$ is a triangle-free member of $\mathcal{N}$, then $A(N)\backslash p \in \mathcal{N}$.
\end{lemma}    
\begin{proof}
    Suppose $N$ is a smallest-rank triangle-free member of $\mathcal{N}$ such that $A(N)\backslash p$ is not in $\mathcal{N}$. By \cite[Lemma 2.3 (i)]{whittle}, each element $e$ of $E(A(N))$ is contained in a hyperplane $H$ isomorphic to $N$. Hence $(A(N)\backslash p)/e$ is isomorphic to $A(N/e)$, and, by Lemma~\ref{tipInClass}, $A(N/e)$ is in $\mathcal{N}$. We deduce that $A(N)\backslash p$ has $M(K_4)$ as an induced restriction. However, as $N$ has no triangles, $A(N)\backslash p$ has no triangles. Hence $A(N)\backslash p$ does not have $M(K_4)$ as an induced restriction. 
\end{proof}

On combining Lemmas~\ref{K4GPC},~\ref{tipInClass},~and~\ref{tiplessInClass}, we immediately obtain the following result. 
\begin{corollary}
\label{cor1}
    The class $\mathcal{N}$ of matroids that do not  have $M(K_4)$ as an induced minor is closed under the following operations.
    \begin{enumerate}[label=(\roman*)]
        \item Generalized parallel connections across projective geometries;
        \item tipped coning; and
        \item tipless coning of triangle-free matroids.
    \end{enumerate}
\end{corollary}

The next lemma \cite[Corollary 3.7]{ox2} will be used in the proofs of the two subsequent lemmas. 

\begin{lemma}
\label{ox1987}
Let $M$ be a $3$-connected binary matroid having rank and corank at least three and suppose that $\{x,y,z\} \subseteq E(M)$. Then $M$ has a minor isomorphic to $M(K_4)$ whose ground set contains $\{x,y,z\}$.
\end{lemma}

\begin{lemma}\label{notriad}
    No $3$-connected member of $\mathcal{N}$ has a triad.
\end{lemma}
\begin{proof}
    Let $T^*$ be a triad of a member $M$ of $\mathcal{N}$. Then, by Lemma~\ref{ox1987},    $T^*$ is  in an $M(K_4)$-minor of $M$. As 
    $(T^*,E(M)-T^*)$ is a 3-separation of $M$, such a minor is obtained by contracting elements from $E(M)-T^*$, so this $M(K_4)$-minor is an induced minor of $M$, a contradiction.
\end{proof}

Recall that, for a rank-$r$ binary matroid $M$ that is viewed as a restriction of $P_r$, if $X\subseteq E(P_r)-E(M)$, we denote by $M+X$ the matroid $P_r|(E(M)\cup X)$. 

\begin{lemma}\label{fillin3sep}
Let $(X,Y)$ be a vertical $3$-separation in a $3$-connected matroid $M$. Let $G_P = \cl_{P_r}(X) \cap \cl_{P_r}(Y)$ and
$G_M = \cl_{M}(X) \cap \cl_{M}(Y)$. If $|G_M| \ge 1$, then $M$ has both
$(M|\cl_M(X))+G_P$ and $(M|\cl_M(Y))+G_P$ as induced minors.
\end{lemma}
\begin{proof}
    Let $G_P=\{a,b,c\}$. Then $G_M \subseteq G_P$ and we may assume that $a \in G_M$. 
  Let $M_X = (M|\cl_M(X))+ G_P$ and $M_Y = (M|\cl_M(Y))+ G_P$. Then it is well known and straightforward to check that each of $M_X$ and $M_Y$ are $3$-connected. Thus, by Lemma~\ref{ox1987}, $\{a,b,c\}$ is contained in an $M(K_4)$-minor of $M_X$. Hence, $M_X$ has an induced minor $N$ that contains $\{a,b,c\}$ and is isomorphic to $M(K_4)$ or $F_7$. Then $N$ has an element $x$ such that $N/x$ has $\{a,b,c\}$ as a triangle and has elements in parallel with each of $b$ and $c$. It follows that $M_Y$ is an induced minor of $M$, and the lemma follows by symmetry. 
\end{proof}

For $r \ge 3$, a \textit{binary $r$-spike with tip $t$} is the vector matroid $M_r$ of the binary matrix $[I_r|J_r - I_r|\bf{1}]$ where $J_r$ and $\bf{1}$ are the $r \times r$ and $r \times 1$ matrices of all ones, respectively, and $\bf{1}$ is labeled by $t$. Note that $M_r/t$ is obtained from an $r$-element circuit by replacing every element by two elements in parallel. It is not difficult to check that all matroids of the form $M_r \backslash z$,   where $z \in E(M_r) - \{t\}$, are isomorphic. We call $M_r\backslash z$ a \textit{binary $r$-spike with tip $t$ and cotip $t^*$}, where $\{t,z,t^*\}$ is a triangle of $M_r$. Clearly $M_3 \backslash z \cong M(K_4)$.  Also, one can show that $M_r\backslash z$ is self-dual for all $r$.

For the rest of the paper,  when we write $M/e$, we mean the matroid obtained from $M$ by contracting the element $e$ but without simplifying the result. The simplification of $M/e$ is denoted by $\si(M/e)$. 

\begin{lemma}\label{tipandcotip}
    Let $M$ be a simple matroid having an element $e$ such that $M/e$ has a basis $B$ each element of which is in a $2$-circuit. If $M/e$ has at least one element that is not in a $2$-circuit, then $M$ has an induced minor isomorphic to a spike with a tip and cotip.
\end{lemma}

\begin{proof} 
Let $f$ be an element of $M/e$ that is not in a 2-circuit. Then, in $M/e$, the fundamental circuit of $f$ with respect to $B$ has every element except $f$ in a 2-circuit. Let $C$ be a smallest circuit of $M/e$ for which, with exactly one exception, every element is in a $2$-circuit. Let $g$ be the exceptional element of $C$. Then the minimality of $C$ implies that 
$M|\cl_M(C\cup \{e\})$ is a spike with tip $e$ and cotip $g$.  
\end{proof}

\begin{corollary}\label{tipcotipK4}
    Let $M$ be a rank-$r$ matroid with $r\geq 3$. If $M$ has an element $e$ such that $M/e$ has a basis, each element of which is in a 
    $2$-circuit, and $M/e$ has at least one element that is not in a $2$-circuit, then $M$ has $M(K_4)$ as an induced minor.
\end{corollary}

The rest of this section is devoted to proving that, when a $3$-connected matroid $M$ has an element $e$ such that neither $M$ nor $M \backslash e$ has $M(K_4)$ as an induced minor, $M$ is a coning with tip $e$ of a triangle-free matroid.

\setcounter{theorem}{12}

\begin{proof}[Proof of Theorem~\ref{trianglethroughe}.]    We argue by induction of $r(M)$. The result is vacuously true when $r(M)=3$. Assume the result holds for $r(M)\leq r-1$ and let $M$ be a matroid satisfying the hypothesis and having $r(M)=r$. We may assume there is an element $f$ such that $\cl_M(\{e,f\})=\{e,f\}$. Let $t$ denote the point of $P_r-E(M)$ on the line spanned by $\{e,f\}$. 
    \begin{sublemma}\label{Mf3conn} $\si(M/f)$ is $3$-connected.
    \end{sublemma}
    Suppose $\si(M/f)$ is not 3-connected. Then $M$ has a vertical 3-separation $(X,Y)$ with $f\in \cl(X) \cap \cl(Y)$. 
    Let $G_P = \cl_{P_r}(X) \cap \cl_{P_r}(Y)$. Suppose first that 
 $e\not \in G_P$. Then, without loss of generality, $e\in X$. By Lemma~\ref{fillin3sep}, $M$ has $(M|\cl_M(X))+G_P$ as an induced minor. By the choice of $M$, every element of $(M|\cl_M(X))+G_P$ is in a triangle with $e$. In particular, every element of the triangle $G_P$ is in  a triangle with $e$. Therefore, $M\backslash e$ has $M(K_4)$ as an induced minor, a contradiction. Hence, we may assume that $e \in G_P$. Let $M_1=M|\cl(X)$ and $M_2=M|\cl(Y)$. By Lemma~\ref{fillin3sep},  $M_1+t$ is an induced minor $M'_1$ of $M$. Since $M_1'$ is 3-connected and neither $M_1'$ nor $M_1'\backslash t$ has $M(K_4)$ as an induced minor, the choice of $M$ implies that  every element of $M_1$ lies in a triangle of $M_1 + t$ with $t$. By symmetry, every element of $M_2$ lies in a triangle of $M_2 + t$ with $t$. 
 
 We show next that every element of $E(M_2)-\{e,f\}$ is in a triangle of $M_2$ with $e$. Assume that an element $y$ of 
$E(M_2)-\{e,f\}$ is not in such a triangle. 
Then $M/y$ is $M_1'$ with no element parallel to $e$, and $M_1'$ is a proper induced minor of $M$. We deduce that neither $M_1'$ nor $M_1'\backslash e$ has $M(K_4)$ as an induced minor. Hence every element of $M_1+t$ is in a triangle with $e$. Let $a$ be an element  of $E(M_1)-\{e,f\}$. Then $M_1$ has elements $b$ and $c$ such that $\{a,b,e\}$ and $\{a,c,t\}$ are triangles of $M_1+t$. Thus, there is an additional element $d$ of $M_1$ that is on the line $\cl(\{c,e\})$. Hence $\cl_M(\{e,f,a\})$ is a rank-3 flat containing six points, a contradiction. We conclude that every element of $E(M_2) - \{e,f\}$ is in a triangle with $e$. Thus, $M_2/e$ has a basis, each element of which is in a 2-circuit, and $f$ is not in a 2-circuit of $M_2/e$. By Corollary~\ref{tipcotipK4},  $M_2$ has $M(K_4)$ as an induced minor. Hence \ref{Mf3conn} holds.

    Let $M_f = \si(M/f)$. Then, by \ref{Mf3conn}, $M_f$ is a 3-connected induced minor of $M$, and neither $M_f$ nor $M_f\backslash e$ has $M(K_4)$ as an induced minor. By the choice of $M$, every element of $M_f$ is in a triangle with $e$. For a triangle $\{x,y,e\}$ of $M_f$, either $\{x,y,e\}$ or $\{x,y,e,f\}$ is a circuit of $M$. From this, we deduce the following. 
    
\begin{sublemma}\label{trilots} 
Every element of $M$ is in a triangle of $M+t$ with at least one of $e$ and $t$. 
\end{sublemma}

Next we prove the following assertion.

    \begin{sublemma}\label{noeandtmeet}
 In $M+t$, among the triangles other than $\{e,f,t\}$, no triangle containing $e$ meets a triangle containing $t$. 
    \end{sublemma}
    
To see this, first suppose there are elements $a,b$, and $c$ in $E(M)$ such that $\{a,b,e\}$ and $\{a,c,t\}$ are distinct triangles in $M+t$. Then, as $M$ has no $M(K_4)$ induced minor, there an element $s$ of $P_r - E(M + t)$ in the rank-3 projective flat $F'$ spanned by $\{e,f,a\}$. 
    
Assume there is a triangle $\{x,y,t\}$  in $M+t$ that is not contained in $F'$. As neither $\si(M/x)$ nor $\si(M/y)$ has $M(K_4)$ as a flat, $M$ has elements 
$z$ and $w$ such that $\{x,w,s\}$ and $\{y,z,s\}$ are  triangles of $M+s$. The projective flat $F$ spanned by $\{e,f,a,b,c,x,y,z,w\}$ 
has rank $4$. Its restriction to the  nine listed points is isomorphic to  the unique single-element extension of $AG(3,2)$.

We show next that $E(M) \cap F = \{a,b,c,e,f,w,x,y,z\}$. Assume this fails. Then, as neither  $s$ nor $t$  is in $E(M)$, the matroid $M|(F\cap E(M))$ must be isomorphic to the complement in $P_4$ of one of  $U_{2,2}$, $U_{2,3}$, $U_{3,3}$, $M(C_4)$, $U_{1,1}\oplus U_{2,3}$,   or the single-element deletion of $M(K_4)$.
In each case except when the complement is $M(C_4)$, the matroid $M|(F\cap E(M))$ has  a flat isomorphic to $M(K_4)$, a contradiction. 
In the exceptional case, $(M\backslash e)|(F\cap E(M\backslash e))$ is isomorphic to the complement in $P_4$ of $M(C_4) \oplus U_{1,1}$ or of the single-element deletion of $M(K_4)$. In these two cases, $(M\backslash e)|(F\cap E(M\backslash e))$
 has  a flat isomorphic to $M(K_4)$, a contradiction. 
We conclude that $E(M) \cap F = \{a,b,c,e,f,w,x,y,z\}$. 

Let $M' = M|\{a,b,c,e,f,w,x,y,z\}$.  
Then $M'/x$ is an induced minor of $M$. This matroid is obtained from  a copy of $F_7$ with ground set $\{a,b,c,e,f,w,x,y\}$ by adding $z$ in parallel to $b$. In particular, in $M'/x$, the element 
$e$ is in a parallel class of size one, so $\si(M'/x\backslash e)$ is an induced minor of $M\backslash e$ that is isomorphic to $M(K_4)$, a contradiction. 
We deduce that every triangle of $M+t$ containing $t$ is contained in $F'$. 

We now know that  every element of $E(M) - \{c,f\}$   is in a triangle of $M$ containing $e$. As $M$ is 3-connected, it has $\{c,f\}$ as a  coindependent set.   Hence $M/e$  has a basis that avoids $\{c,f\}$. Therefore, $M/e$ has a basis  each element of which is in a 2-circuit. As $c$ is an element of $M/e$ that is not in a 2-circuit,   Corollary~\ref{tipcotipK4} implies that $M$ has $M(K_4)$ as an induced minor, a contradiction. 
We conclude that \ref{noeandtmeet} holds.

  \begin{sublemma}\label{noftriangle}  
  $M$ has no triangle containing $f$. 
  \end{sublemma}
  
Assume that $M$ has such a triangle $T$. Then one easily checks using \ref{trilots} that, in $M+t$, the flat spanned by $T \cup \{e\}$ contains distinct intersecting triangles other than $\{e,f,t\}$, one containing $e$ and one containing $t$,  a contradiction to  \ref{noeandtmeet}.

Let $\{g,h,t\}$ be a triangle of $M+t$ different from $\{e,f,t\}$. Then, by \ref{noeandtmeet} and \ref{noftriangle}, both 
$\cl_M(\{e,g\})$  and  $\cl_M(\{f,g\})$ contain exactly two elements. Let $u$ and $s$ be the points in $P_r-E(M)$ on the lines $\cl_{P_r}(\{e,g\})$ and $\cl_{P_r}(\{f,g\})$, respectively. Then, by replacing $\{e,f,t\}$  by $\{e,g,u\}$ and by $\{e,h,s\}$, we deduce from \ref{trilots} and \ref{noeandtmeet} that, in $M+u$, every element of $E(M) - g$ is in a triangle with exactly one of $e$ and $u$; and, in $M+s$, every element of $E(M) - h$ is in a triangle with exactly one of $e$ and $s$.

\begin{sublemma}
\label{atmostone}
$M+t$ has at most one triangle other than $\{e,f,t\}$ containing $t$.
\end{sublemma}

Suppose that $\{x,y,t\}$ is a triangle of $M+t$ that differs from both $\{e,f,t\}$ and $\{g,h,t\}$. By \ref{noeandtmeet}, for the triangles $\{e,x,i\}$ and $\{e,y,j\}$ of $P_r$, neither $i$ nor $j$ is in $E(M)$. Thus, $M+u$ has triangles $\{u,x,z\}$ and $\{u,y,w\}$ for some  $z$ and $w$ in $E(M)$. Now $M|\{e,f,g,h,x,y,w,z\}\cong AG(3,2)$. Let $F=\cl_M(\{e,f,g,h,x,y,w,z\}$. Then $\cl_{P_r}(F)-F$ contains $\{i,j,s,t,u\}$ and $P_r|\{i,j,s,t,u\}$ is a single-element deletion of $M(K_4)$. It follows that $M|F$ is one of $AG(3,2)$, the unique rank-4 single-element extension of $AG(3,2)$, or the unique rank-4 2-element extension of $AG(3,2)$. In the last case, one can easily check that, for some $p$ in $\{s,t,u\}$, the matroid $M+p$ has a triangle containing $e$ that meets a triangle containing $p$ and does not contain $\{e,p\}$, a contradiction to \ref{noeandtmeet}. In the second case, $M|F$ is a rank-4 binary spike with a tip and this tip is not equal to $e$. In that and the first case taking $a$ to be an element of $F$ that differs from $e$ and from the spike tip when it is present, we have that $\si((M|F)/a\backslash e)\cong M(K_4)$, a contradiction. Hence \ref{atmostone} holds.

We now know that either every element of $M$ except $f$ is in a triangle with $e$, or  $M$ has exactly three elements $f$, $g$, and $h$ that are  not in triangles with $e$. In the first case, $M/e$ certainly has a basis  each element of which is in a 2-circuit. As $\{e,f\}$ is a flat of $M$,  by Corollary~\ref{tipcotipK4}, $M$ has an induced minor isomorphic to $M(K_4)$, a contradiction. 
We deduce that $f$, $g$, and $h$  are 
  the only elements of $E(M) - e$ that are not in a triangle with $e$. By Lemma~\ref{notriad}, $M$ has no triads.  Hence $r(M) = r(M\backslash \{f,g,h\})$. Thus, $M/e$  has a basis avoiding $\{f,g,h\}$. Each element of this basis is in a 2-circuit of $M/e$. Also, $f$ is in a parallel class of size one in $M/e$. Thus, by Corollary~\ref{tipcotipK4}, $M$ has $M(K_4)$ as an induced minor, a contradiction. 
    \end{proof}

\end{document}